\documentclass[preprint,12pt]{elsarticle}

\usepackage{hyperref}
\usepackage{amssymb}
\usepackage{amsfonts}
\usepackage{amsmath}

\newtheorem{theorem}{Theorem}

\newtheorem{corollary}[theorem]{Corollary}

\newtheorem{definition}[theorem]{Definition}

\newtheorem{remark}[theorem]{Remark}

\newenvironment{proof}[1][Proof]{\noindent\textbf{#1.} }{\ \rule{0.5em}{0.5em}}

\begin{document}

\begin{frontmatter}

\title{A note on some identities involving special functions from
the hypergeometric solution of algebraic equations
}

\author{J. L. Gonz\'{a}lez-Santander\fnref{myfootnote}}
\address{Department of Mathematics. Universidad de Oviedo. C/\ Federico Garc\'{\i}a Lorca, 18. 33007 Oviedo, Asturias. Spain.}
\fntext[myfootnote]{gonzalezmarjuan@uniovi.es}




\begin{abstract}
From the algebraic solution of $x^{n}-x+t=0$ for $n=2,3,4$ and the
corresponding solution in terms of hypergeometric functions, we obtain a set
of reduction formulas for hypergeometric functions. By differentiation and
integration of these results, and applying other known reduction formulas of
hypergeometric functions, we derive new reduction formulas of special
functions as well as the calculation of some infinite integrals in terms of
elementary functions.
\end{abstract}


\end{frontmatter}


\section{Introduction and preliminaries}

In the literature, we found a large body of literature dealing with the
trinomial equation. In fact, there are different versions of this kind of
equation. For instance, in 1915, Mellin studied the trinomial equation \cite%
{Belardinelli}:
\begin{equation}
y^{n}+x\,y^{p}+1=0,\qquad n>p,  \label{Trinomial_Mellin}
\end{equation}%
where $n,p$ are positive integers and $x\in
\mathbb{R}
$. By using his integral transform, Mellin derived the following series
representation \cite[Chap.3 Sect.8]{Hochstadt}:%
\begin{equation*}
y\left( x\right) =\frac{1}{n}\sum_{r=0}^{\infty }\frac{\,\Gamma \left( \frac{%
1+pr}{n}\right) }{\,\Gamma \left( \frac{1+pr}{n}+1-r\right) }\frac{\left(
-x\right) ^{r}}{r!},
\end{equation*}%
where $\left\vert x\right\vert <\left( p/n\right) ^{-p/n}\left( 1-p/n\right)
^{p/n-1}\leq 2$. In \cite{Miller}, Miller rewrote the solution of (\ref%
{Trinomial_Mellin})\ for positive real numbers $n,p$ in terms of the Wright
function.

In this paper, we are interested in the trinomial equation:
\begin{equation}
x^{n}-x+t=0,  \label{Main_Eqn}
\end{equation}%
assuming that $n\geq 2$ is an integer and $t\in
\mathbb{C}
$. Performing the change of variables $z=t^{-1/n}x$ and $a=-t^{-1/n-1}$, (%
\ref{Main_Eqn})\ is transformed into:
\begin{equation}
z^{n}+a\,z+1=0.  \label{Bring_form}
\end{equation}

Bring and Jerrard proved that any fifth degree equation can be brougth to
into the form given in (\ref{Bring_form})\ for $n=5$ by means of Tschirnhaus
transformations \cite{Adamchik}. Nonetheless, Abel proved in 1824 the
impossibility of solving the general quintic equation by means of radicals
\cite{Pessic}. Consequently, mathematicians turned to the idea of searching
for analytic solutions. The first success in this direction was achieved in
1858 by Hermite and Kronecker who were able to express the solution to the
quintic equation by means of a modular elliptic function (see \cite{Klein}).

Equation (\ref{Main_Eqn})\ was first solved by Lambert in 1758 as a series
development for $x$ in powers of $t$ \cite{Lambert}. Euler's version of
Lambert series \cite{Euler} is connected to the \textit{tree function} and
the \textit{Lambert} $W$ \textit{function} \cite{Corless}. More recently,
Glasser calculated the roots of (\ref{Main_Eqn}) as a finite sum of
generalized hypergeometric functions \cite{Glasser}. In many cases, one of
the roots can be expressed as a single hypergeometric function. However, in
1770, Lagrange \cite{Lagrange} applied his inversion formula \cite[Appendix E%
]{Andrews} to derive a root $x_{n}\left( t\right) $ of the equation (\ref%
{Main_Eqn})\ as an expansion in powers of $t$. Next, we present the
derivation given by Lagrange.

\begin{theorem}[Lagrange inversion formula]
Consider the variables $x$, $t$, and $r$ related by
\begin{equation}
x=t+r\,\phi \left( x\right) ,  \label{x=t+r*phi}
\end{equation}%
where $\phi \left( x\right) $ is analytic in the neighborhood of $x=t$ with $%
\phi \left( t\right) \neq 0$. Consider as well an analytic function $f\left(
x\right) $ in the neighborhood of $x=t$. Then Lagrange's formula is%
\begin{equation}
f\left( x\right) =f\left( t\right) +\sum_{k=1}^{\infty }\frac{r^{k}}{k!}%
\frac{d^{k-1}}{dt^{k-1}}\left[ f^{\prime }\left( t\right) \phi ^{k}\left(
t\right) \right] .  \label{f(x)_Lagrange_inv}
\end{equation}
\end{theorem}

If we take $f\left( x\right) =x$, $\phi \left( x\right) =x^{n}$, and $r=1$, (%
\ref{x=t+r*phi}) becomes (\ref{Main_Eqn}),\ and (\ref{f(x)_Lagrange_inv})\
reads as%
\begin{equation}
x_{n}\left( t\right) =x=t\left[ 1+\sum_{k=1}^{\infty }\frac{\left( nk\right)
!}{k!\left( nk-k+1\right) !}t^{\left( n-1\right) k}\right] .
\label{Root_Lagrange}
\end{equation}

Fortunately, we can recast (\ref{Root_Lagrange})\ in hypergeometric form.
For this purpose, denote the Pochhammer symbol as $\left( a\right)
_{n}=\Gamma \left( a+n\right) /\Gamma \left( a\right) $, where $\Gamma
\left( z\right) $ is the gamma function. Then, we can prove easily by
induction that:
\begin{equation}
\left( nk\right) !=\left( \frac{1}{n}\right) _{k}\left( \frac{2}{n}\right)
_{k}\cdots \left( \frac{n}{n}\right) _{k}\,n^{nk},  \label{(nk)!}
\end{equation}%
hence%
\begin{equation}
\left( nk-k+1\right) !=\left( \frac{2}{n-1}\right) _{k}\cdots \left( \frac{n%
}{n-1}\right) _{k}\left( n-1\right) ^{\left( n-1\right) k}.
\label{(nk-k+1)!}
\end{equation}

Now, let us define the generalized hypergeometric series\textit{\footnote{%
For the different cases of convergence of the generalized hypergeometric
series see \cite[Sect. 16.2]{DLMF}. }} as follows:

\begin{definition}[Generalized hypergeometric series]
\begin{equation}
_{p}F_{q}\left( \left.
\begin{array}{c}
a_{1},\ldots ,a_{p} \\
b_{1},\ldots ,b_{q}%
\end{array}%
\right\vert z\right) =\sum_{k=0}^{\infty }\frac{\left( a_{1}\right)
_{k}\cdots \left( a_{p}\right) _{k}}{\left( b_{1}\right) _{k}\cdots \left(
b_{q}\right) _{k}}\frac{z^{k}}{k!}.  \label{Hyper_Generalized_def}
\end{equation}%
If none of the parameters $b_{1},\ldots b_{q}$ are nonpositive integers and $%
p\leq q$, the series (\ref{Hyper_Generalized_def})\ converges for all finite
values of $z$ and defines and entire function.
\end{definition}

\begin{remark}
Note that
\begin{equation}
\lim_{z\rightarrow 0}\,_{p}F_{q}\left( \left.
\begin{array}{c}
a_{1},\ldots ,a_{p} \\
b_{1},\ldots ,b_{q}%
\end{array}%
\right\vert z\right) =1.  \label{Hyper_z=1}
\end{equation}
\end{remark}

Therefore, according to (\ref{(nk)!})-(\ref{Hyper_Generalized_def}), we
finally rewrite (\ref{Root_Lagrange})\ as%
\begin{equation}
x_{n}\left( t\right) =t\,_{n}F_{n-1}\left( \left.
\begin{array}{c}
\frac{1}{n},\ldots ,\frac{n}{n} \\
\frac{2}{n-1},\ldots ,\frac{n}{n-1}%
\end{array}%
\right\vert n\left( \frac{nt}{n-1}\right) ^{n-1}\right) .
\label{x_n(t)_general}
\end{equation}

On the one hand, for $n=2,3,4$, (\ref{x_n(t)_general})\ is reduced to (see
as well \cite{Perelomov}): \
\begin{eqnarray}
x_{2}\left( t\right) &=&t\,_{2}F_{1}\left( \left.
\begin{array}{c}
\frac{1}{2},1 \\
2%
\end{array}%
\right\vert 4t\right) ,  \label{x2(t)_Hyp} \\
x_{3}\left( t\right) &=&t\,_{2}F_{1}\left( \left.
\begin{array}{c}
\frac{1}{3},\frac{2}{3} \\
\frac{3}{2}%
\end{array}%
\right\vert 3\left( \frac{3t}{2}\right) ^{2}\right) ,  \label{x3(t)_Hyp} \\
x_{4}\left( t\right) &=&t\,_{3}F_{2}\left( \left.
\begin{array}{c}
\frac{1}{4},\frac{1}{2},\frac{3}{4} \\
\frac{2}{3},\frac{4}{3}%
\end{array}%
\right\vert 4\left( \frac{4t}{3}\right) ^{3}\right) .  \label{x4(t)_Hyp}
\end{eqnarray}%
where we have simplified the hypergeometric series with common Pochhammer
symbols in numerator and denominator in (\ref{x3(t)_Hyp})\ and (\ref%
{x4(t)_Hyp}), according to definition (\ref{Hyper_Generalized_def}). On the
other hand, it is well-known that the roots of (\ref{x2(t)_Hyp})-(\ref%
{x4(t)_Hyp})\ are expressible in terms of elementary functions. The scope of
this paper is just to compare both approaches, and from this comparison,
derive some new reduction formulas and definite integrals involving special
functions. For this purpose, we will use the following differentiation
formulas, that can be easily proved by induction:\
\begin{eqnarray}
\frac{d^{n}}{dt^{n}}\left( \frac{1}{t}\right) &=&\frac{\left( -1\right)
^{n}n!}{t^{n+1}},  \label{Dn(1/t)} \\
\frac{d^{n}}{dt^{n}}\left( \sqrt{1-t}\right) &=&\left( -\frac{1}{2}\right)
_{n}\,\left( 1-t\right) ^{1/2-n}.  \label{Dn(sqrt(1-t))}
\end{eqnarray}

Notice that for $n=1$ in (\ref{Dn(sqrt(1-t))}), we have that
\begin{equation*}
\frac{d}{dt}\left( \sqrt{1-t}\right) =\frac{-1}{2\sqrt{1-t}},
\end{equation*}%
thus, knowing that \cite[Eqn. 18:5:7]{Atlas}
\begin{equation}
\left( x\right) _{n+1}=x\left( x+1\right) _{n},  \label{(x)_n+1}
\end{equation}%
we conclude%
\begin{equation}
\frac{d^{n}}{dt^{n}}\left( \frac{1}{\sqrt{1-t}}\right) =\left( \frac{1}{2}%
\right) _{n}\,\left( 1-t\right) ^{-1/2-n}.  \label{Dn(1/sqrt)}
\end{equation}

Also, we will use Leibniz's differentiation formula \cite[Eqn. 1.4.2]{DLMF}
(for the historical origin of this formula, see \cite[p. 143]{Leibniz}),%
\begin{equation}
\frac{d^{n}}{dt^{n}}\left[ f\left( t\right) g\left( t\right) \right]
=\sum_{k=0}^{n}\binom{n}{k}f^{\left( k\right) }\left( t\right) g^{\left(
n-k\right) }\left( t\right) ,  \label{Leibniz_formula}
\end{equation}%
Gauss summation formula \cite[Eqn. 15.4.20]{DLMF} (for the original work of
Gauss, see \cite{Gauss}),
\begin{eqnarray}
&&_{2}F_{1}\left( \left.
\begin{array}{c}
a,b \\
c%
\end{array}%
\right\vert 1\right) =\frac{\Gamma \left( c\right) \Gamma \left(
c-a-b\right) }{\Gamma \left( c-a\right) \Gamma \left( c-b\right) },
\label{Gauss_sum} \\
&&\mathrm{Re}\,\left( c-a-b\right) >0,  \notag
\end{eqnarray}%
and Whipple's sum \cite[Eqn. 16.4.7]{DLMF},
\begin{eqnarray}
&&_{3}F_{2}\left( \left.
\begin{array}{c}
a,1-a,c \\
d,2c-d+1%
\end{array}%
\right\vert 1\right)  \label{Whipple_sum} \\
&=&\frac{\pi \,2^{1-2c}\,\Gamma \left( d\right) \Gamma \left( 2c-d+1\right)
}{\Gamma \left( c+\frac{a-d+1}{2}\right) \Gamma \left( c+1-\frac{a+d}{2}%
\right) \Gamma \left( \frac{a+d}{2}\right) \Gamma \left( \frac{d-a+1}{2}%
\right) },  \notag \\
&&\mathrm{Re}\,\left( c\right) >0\text{ or }a\in
\mathbb{Z}
.  \notag
\end{eqnarray}

For the calculation of the definite integrals, we will use the following
result \cite[Ch. 2. Ex. 11]{Andrews}:
\begin{eqnarray}
&&\int_{0}^{\infty }e^{-st}t^{\alpha -1}\,_{p}F_{q}\left( \left.
\begin{array}{c}
a_{1},\ldots ,a_{p} \\
b_{1},\ldots ,b_{q}%
\end{array}%
\right\vert xt\right) dt  \label{Int_Andrews} \\
&=&\frac{\Gamma \left( \alpha \right) }{s^{\alpha }}\,_{p+1}F_{q}\left(
\left.
\begin{array}{c}
a_{1},\ldots ,a_{p},\alpha \\
b_{1},\ldots ,b_{q}%
\end{array}%
\right\vert \frac{x}{s}\right) ,  \notag \\
&&p\leq q,\,\mathrm{Re}\,s>0,\,\mathrm{Re}\,\alpha >0.  \notag
\end{eqnarray}

This paper is organized as follows. Section $2$ equates the solution of (\ref%
{Main_Eqn}) for $n=2$ to $x_{2}\left( t\right) $. From this result, and
using some differentiation formulas of the hypergeometric $_{2}F_{1}$
function, we obtain a set of reduction formulas of some hypergeometric
functions in terms of elementary functions, which extends the classical
Schwarz's list \cite{Schwarz}. As corollaries, we obtain identities
involving the incomplete beta function and the Legendre function. Also, we
calculate two infinite integrals involving the lower incomplete gamma
function. Section $3$ equates the solution of (\ref{Main_Eqn}) for $n=3$ to $%
x_{3}\left( t\right) $, and from it, we derive a new reduction formula of an
hypergeometric $_{2}F_{1}$ function in terms of elementary functions. Also,
we calculate an infinite integral involving the parabolic cylinder function.
Section $4$ derives a reduction formula of a $_{3}F_{2}$ function in terms
of elementary functions, equating the solution of (\ref{Main_Eqn}) for $n=4$
to $x_{4}\left( t\right) $. From the latter reduction formula, we obtain an
identity involving the product of two Legendre functions. Finally, Section $5
$ collects our conclusions. In the Appendix, we recall the solution of the
cubic and the quartic equations.

\section{Case $n=\ 2$}

In this case, the algebraic solution of (\ref{Main_Eqn})\ is
\begin{equation}
x_{2}\left( t\right) =\frac{1\pm \sqrt{1-4t}}{2},  \label{x2(t)_alg}
\end{equation}%
hence, selecting the proper root of (\ref{x2(t)_alg}), we can equate it to (%
\ref{x2(t)_Hyp}), to obtain%
\begin{equation}
_{2}F_{1}\left( \left.
\begin{array}{c}
\frac{1}{2},1 \\
2%
\end{array}%
\right\vert t\right) =\frac{2}{t}\left( 1-\sqrt{1-t}\right) ,
\label{2F1_resultado}
\end{equation}%
which agrees with the result reported in the literature \cite[Eqn. 7.3.2(84)]%
{Prudnikov3}. Notice that (\ref{2F1_resultado})\ can be derived from the
binomial theorem \cite[Eqn. 6.14.1]{Atlas}. Indeed,
\begin{equation*}
\sqrt{1-t}=\sum_{k=0}^{\infty }\left( -\frac{1}{2}\right) _{k}\,\frac{t^{k}}{%
k!}=1+\sum_{k=1}^{\infty }\left( -\frac{1}{2}\right) _{k}\,\frac{t^{k}}{k!},
\end{equation*}%
hence, applying (\ref{(x)_n+1}), we have%
\begin{equation*}
\frac{\sqrt{1-t}-1}{t}=\sum_{k=1}^{\infty }\left( -\frac{1}{2}\right) _{k}\,%
\frac{t^{k-1}}{k!}=-\frac{1}{2}\sum_{k=1}^{\infty }\left( \frac{1}{2}\right)
_{k-1}\,\frac{\left( 1\right) _{k-1}\,t^{k-1}}{\left( 2\right) _{k-1}\left(
k-1\right) !},
\end{equation*}%
and the result follows. From (\ref{2F1_resultado}), we obtain next a set of
results using the formulas stated in the Introduction.

\subsection{First differentiation formula}

\begin{theorem}
For $n=0,1,2,\ldots $ and $t\in
\mathbb{C}
$, the following reduction formula holds true:%
\begin{eqnarray}
&&\,_{2}F_{1}\left( \left.
\begin{array}{c}
\frac{1}{2}+n,1+n \\
2+n%
\end{array}%
\right\vert t\right)  \label{2F1_resultado_(1)} \\
&=&\left\{
\begin{array}{ll}
\displaystyle%
\frac{2\left( -1\right) ^{n}\left( n+1\right) !}{\left( \frac{1}{2}\right)
_{n}\,t^{n+1}}\left[ 1-\sqrt{1-t}\sum_{k=0}^{n}\frac{\left( -\frac{1}{2}%
\right) _{k}}{k!}\left( \frac{t}{t-1}\right) ^{k}\right] , & t\neq 0,1, \\
1, & t=0, \\
2, & t=1,n=0, \\
\infty , & t=1,n\geq 1.%
\end{array}%
\right.  \notag
\end{eqnarray}
\end{theorem}

\begin{proof}
For $t\neq 0,1$, apply the following differentiation formula for the Gauss
hypergeometric function \cite[Eqn. 15.5.2]{DLMF}:\
\begin{eqnarray}
&&\frac{d^{n}}{dt^{n}}\left[ _{2}F_{1}\left( \left.
\begin{array}{c}
a,b \\
c%
\end{array}%
\right\vert t\right) \right] =\frac{\left( a\right) _{n}\left( b\right) _{n}%
}{\left( c\right) _{n}}\,_{2}F_{1}\left( \left.
\begin{array}{c}
a+n,b+n \\
c+n%
\end{array}%
\right\vert t\right) .  \label{Dn_(1)} \\
&&n=0,1,2,\ldots  \notag
\end{eqnarray}%
Therefore, taking $a=\frac{1}{2}$, $b=1$ and $c=2$ in (\ref{Dn_(1)}) and
using (\ref{2F1_resultado}), we have
\begin{equation*}
\frac{d^{n}}{dt^{n}}\left[ _{2}F_{1}\left( \left.
\begin{array}{c}
\frac{1}{2},1 \\
2%
\end{array}%
\right\vert t\right) \right] =2\left[ \frac{d^{n}}{dt^{n}}\left( \frac{1}{t}%
\right) -\frac{d^{n}}{dt^{n}}\left( \frac{\sqrt{1-t}}{t}\right) \right] .
\end{equation*}%
Applying (\ref{Dn(1/t)})-(\ref{Dn(sqrt(1-t))}) and (\ref{Leibniz_formula}),
after some algebra, we arrive at (\ref{2F1_resultado_(1)}) for $t\neq 0,1$.

For $t=0$, apply (\ref{Hyper_z=1}).

For $t=1$, apply Gauss summation formula (\ref{Gauss_sum}). This completes
the proof.
\end{proof}

\begin{corollary}
For $n=0,1,2,\ldots $ and $t\in
\mathbb{C}
$, the following reduction formula holds true:%
\begin{eqnarray}
&&\,_{2}F_{1}\left( \left.
\begin{array}{c}
1+2n,\frac{3}{2}+n \\
3+2n%
\end{array}%
\right\vert t\right)   \label{2F1_quadratic} \\
&=&\left\{
\begin{array}{ll}
\begin{array}{l}
\displaystyle%
\frac{\left( -1\right) ^{n}\left( n+1\right) !}{\left( \frac{1}{2}\right)
_{n}\,}\left( \frac{2}{t}\right) ^{2\left( n+1\right) } \\
\displaystyle%
\left[ 2-t-2\sqrt{1-t}\sum_{k=0}^{n}\frac{\left( -\frac{1}{2}\right) _{k}}{k!%
}\left( \frac{t}{2\sqrt{t-1}}\right) ^{2k}\right] ,%
\end{array}
& t\neq 0,1 \\
1, & t=0, \\
4, & t=1,n=0, \\
\infty , & t=1,n\geq 1.%
\end{array}%
\right.   \notag
\end{eqnarray}
\end{corollary}

\begin{proof}
Apply the quadratic transformation \cite[Eqn. 3.1.7]{Andrews}:%
\begin{equation*}
_{2}F_{1}\left( \left.
\begin{array}{c}
\alpha ,\beta \\
2\beta%
\end{array}%
\right\vert x\right) =\left( 1-\frac{x}{2}\right) ^{-\alpha
}\,_{2}F_{1}\left( \left.
\begin{array}{c}
\frac{\alpha }{2},\frac{\alpha +1}{2} \\
\beta +\frac{1}{2}%
\end{array}%
\right\vert \left( \frac{x}{2-x}\right) ^{2}\right) ,
\end{equation*}%
taking $\alpha =2n+1$, $\beta =n+\frac{3}{2}$, and $x=\frac{2\sqrt{z}}{1+%
\sqrt{z}}$ to arrive at:%
\begin{equation}
\,_{2}F_{1}\left( \left.
\begin{array}{c}
1+2n,\frac{3}{2}+n \\
3+2n%
\end{array}%
\right\vert \frac{2\sqrt{z}}{1+\sqrt{z}}\right) =\left( 1+\sqrt{z}\right)
^{2n+1}\,_{2}F_{1}\left( \left.
\begin{array}{c}
\frac{1}{2}+n,1+n \\
2+n%
\end{array}%
\right\vert z\right) .  \label{Quadratic}
\end{equation}%
In order to obtain the desired result for $t\neq 0,1$, substitute (\ref%
{2F1_resultado_(1)})\ in (\ref{Quadratic})\ and perform the change of
variables $t=\frac{2\sqrt{z}}{1+\sqrt{z}}$. According to this last result,
the cases given in (\ref{2F1_quadratic})\ for $t=1$ are straightforward.
However, for $t=0$, we have an indeterminate expression on the RHS\ of (\ref%
{2F1_quadratic}). On the one hand, according to (\ref{Hyper_z=1}), we have%
\begin{equation*}
\lim_{t\rightarrow 0}\,_{2}F_{1}\left( \left.
\begin{array}{c}
1+2n,\frac{3}{2}+n \\
3+2n%
\end{array}%
\right\vert t\right) =1.
\end{equation*}%
On the other hand, we calculate the limit $t\rightarrow 0$ of the RHS of (%
\ref{2F1_quadratic}), taking into account the formula \cite[Eqn. 18:3:4]%
{Atlas}:%
\begin{equation}
\frac{1}{\left( 1-t\right) ^{\nu }}=\sum_{k=0}^{\infty }\left( \nu \right)
_{k}\frac{t^{k}}{k!},  \label{Binomial}
\end{equation}%
thereby
\begin{eqnarray*}
&&\frac{\left( -1\right) ^{n}\left( n+1\right) !}{\left( \frac{1}{2}\right)
_{n}\,}\lim_{t\rightarrow 0}\left( \frac{2}{t}\right) ^{2\left( n+1\right) }
\\
&&\left\{ 2-t-2\sqrt{1-t}\left[ \sum_{k=0}^{\infty }\frac{\left( -\frac{1}{2}%
\right) _{k}}{k!}\left( \frac{t^{2}}{4\left( t-1\right) }\right)
^{k}-\sum_{k=n+1}^{\infty }\frac{\left( -\frac{1}{2}\right) _{k}}{k!}\left(
\frac{t^{2}}{4\left( t-1\right) }\right) ^{k}\right] \right\} \\
&=&\frac{\left( -1\right) ^{n}\left( n+1\right) !}{\left( \frac{1}{2}\right)
_{n}\,}\lim_{t\rightarrow 0}\left( \frac{2}{t}\right) ^{2\left( n+1\right) }
\\
&&\left\{ 2-t-2\sqrt{1-t}\left[ \frac{2-t}{2\sqrt{1-t}}-\sum_{k=n+1}^{\infty
}\frac{\left( -\frac{1}{2}\right) _{k}}{k!}\left( \frac{t^{2}}{4\left(
t-1\right) }\right) ^{k}\right] \right\} \\
&=&-\frac{2\left( -\frac{1}{2}\right) _{n+1}}{\left( \frac{1}{2}\right)
_{n}\,}=1,
\end{eqnarray*}%
where we have applied (\ref{(x)_n+1})\ for $x=-\frac{1}{2}$.
\end{proof}

\begin{corollary}
For $n=0,1,2,\ldots $ and $t\in
\mathbb{C}
$, the following reduction formula holds true:%
\begin{eqnarray}
&&\mathrm{B}\left( 1+n,\frac{1}{2}-n,t\right)  \label{beta_resultado} \\
&=&\left\{
\begin{array}{ll}
\displaystyle%
\frac{2(-1)^{n}n!}{\left( \frac{1}{2}\right) _{n}}\left[ 1-\sqrt{1-t}%
\sum_{k=0}^{n}\frac{\left( -\frac{1}{2}\right) _{k}}{k!}\left( \frac{t}{t-1}%
\right) ^{k}\right] , & t\neq 0,1, \\
\frac{2(-1)^{n}n!}{\left( \frac{1}{2}\right) _{n}}, & t=1, \\
0, & t=0.%
\end{array}%
\right.  \notag
\end{eqnarray}%
where $\mathrm{B}\left( \nu ,\mu ,z\right) $ denotes the \textit{incomplete
beta function }\cite[Chap. 58]{Atlas}.
\end{corollary}

\begin{proof}
For $t\neq 0,1$, in \cite[Eqn. 7.3.1(28)]{Prudnikov3}, we found:%
\begin{equation}
_{2}F_{1}\left( \left.
\begin{array}{c}
a,b \\
b+1%
\end{array}%
\right\vert t\right) =b\ t^{-b}\ \mathrm{B}\left( b,1-a,t\right) ,
\label{2F1_beta}
\end{equation}%
Therefore, take $a=\frac{1}{2}+n$ and $b=1+n$ in (\ref{2F1_beta})\ and apply
(\ref{2F1_resultado_(1)}) to obtain (\ref{beta_resultado}).

For $t=1$, apply the properties of the incomplete beta function \cite[Eqns.
58:3:1\&58:1:1]{Atlas}%
\begin{equation*}
\mathrm{B}\left( \nu ,\mu ,1\right) =\mathrm{B}\left( \nu ,\mu \right) =%
\frac{\Gamma \left( \nu \right) \Gamma \left( \mu \right) }{\Gamma \left(
\nu +\mu \right) },
\end{equation*}%
and the formula of the gamma function \cite[Eqn. 43:4:4]{Atlas}%
\begin{equation*}
\Gamma \left( \frac{1}{2}-n\right) =\frac{\left( -1\right) ^{n}}{\left(
\frac{1}{2}\right) _{n}}\sqrt{\pi },
\end{equation*}%
to obtain the desired result.

For $t=0$, apply the definition of the incomplete beta function \cite[%
Eqn.58:3:1]{Atlas}, and calculate the limit $t\rightarrow 0$ for $n\geq 0$,
to obtain:%
\begin{equation*}
\lim_{t\rightarrow 0}\mathrm{B}\left( 1+n,\mu ,t\right) =\lim_{t\rightarrow
0}\int_{0}^{t}x^{n}\left( 1-x\right) ^{\mu -1}dx=0.
\end{equation*}%

\end{proof}

It is worth noting that we can derive a different elementary representation
of $_{2}F_{1}\left( \frac{1}{2}+n,1+n;2+n;t\right) $ by using known formulas
given in the literature.

\begin{theorem}
For $n=0,1,2,\ldots $ and $t\in
\mathbb{C}
$, the following reduction formula holds true:%
\begin{eqnarray}
&&\,_{2}F_{1}\left( \left.
\begin{array}{c}
\frac{1}{2}+n,1+n \\
2+n%
\end{array}%
\right\vert t\right)  \label{2F1_resultado_(1)b} \\
&=&\left\{
\begin{array}{ll}
\displaystyle%
\frac{2\left( -1\right) ^{n}\left( n+1\right) !}{\left( \frac{1}{2}\right)
_{n}\,t^{n+1}}\left[ 1-\left( 1-t\right) ^{1/2-n}\sum_{k=0}^{n}\frac{\left(
\frac{1}{2}-n\right) _{k}}{k!}t^{k}\right] , & t\neq 0,1, \\
1, & t=0, \\
2, & t=1,n=0, \\
\infty , & t=1,n\geq 1.%
\end{array}%
\right.  \notag
\end{eqnarray}
\end{theorem}

\begin{proof}
First we prove (\ref{2F1_resultado_(1)b}) for $t\neq 0,1$. Apply Euler's
transformation formula \cite[Eqn. 15.8.1]{DLMF}:
\begin{equation*}
_{2}F_{1}\left( \left.
\begin{array}{c}
\alpha ,\beta \\
\gamma%
\end{array}%
\right\vert z\right) =\left( 1-z\right) ^{\gamma -\alpha -\beta
}\,_{2}F_{1}\left( \left.
\begin{array}{c}
\gamma -\alpha ,\gamma -\beta \\
\gamma%
\end{array}%
\right\vert z\right) ,
\end{equation*}%
to obtain%
\begin{equation}
_{2}F_{1}\left( \left.
\begin{array}{c}
\frac{1}{2}+n,1+n \\
2+n%
\end{array}%
\right\vert t\right) =\left( 1-t\right) ^{1/2-n}\,_{2}F_{1}\left( \left.
\begin{array}{c}
\frac{3}{2},1 \\
2+n%
\end{array}%
\right\vert t\right) .  \label{2F1_(1)b_intermedio}
\end{equation}%
We found in \cite[Eqn. 7.3.1(123)]{Prudnikov3} for $m=1,2,\ldots $, and $%
m-b\neq 1,2,\ldots $, the formula:
\begin{equation}
_{2}F_{1}\left( \left.
\begin{array}{c}
1,b \\
m%
\end{array}%
\right\vert z\right) =\frac{\left( m-1\right) !\left( -z\right) ^{1-m}}{%
\left( 1-b\right) _{m-1}}\left[ \left( 1-z\right) ^{m-b-1}-\sum_{k=0}^{m-2}%
\frac{\left( b-m+1\right) _{k}}{k!}z^{k}\right] .  \label{Formula_Prudnikov}
\end{equation}%
Therefore, apply (\ref{Formula_Prudnikov})\ to (\ref{2F1_(1)b_intermedio})\
with $m=n+2$ and $b=\frac{3}{2}$, taking into account (\ref{(x)_n+1})\ for $%
x=-\frac{1}{2}$, to arrive at (\ref{2F1_resultado_(1)b})\ for $t\neq 0,1$.

Straightforward from (\ref{2F1_resultado_(1)b})\ for $t\neq 0,1$, we have a
divergent result for $t=1$, except for $n=0$.

For $t=0$, we have\ an indeterminate expression on the RHS of (\ref%
{2F1_resultado_(1)b}). On the one hand, according to (\ref{Hyper_z=1}), we
have%
\begin{equation}
\lim_{t\rightarrow 0}\,_{2}F_{1}\left( \left.
\begin{array}{c}
\frac{1}{2}+n,1+n \\
2+n%
\end{array}%
\right\vert t\right) =1.  \label{2F1_t=0}
\end{equation}%
On the other hand, we calculate the limit $t\rightarrow 0$ of the RHS of (%
\ref{2F1_resultado_(1)b}) taking into account (\ref{Binomial}). Thereby
\begin{eqnarray*}
&&\frac{2\left( -1\right) ^{n}\left( n+1\right) !}{\left( \frac{1}{2}\right)
_{n}\,}\lim_{t\rightarrow 0}\frac{\left( 1-t\right) ^{1/2-n}}{t^{n+1}}\left[
\frac{1}{\left( 1-t\right) ^{1/2-n}}-\sum_{k=0}^{n}\frac{\left( \frac{1}{2}%
-n\right) _{k}}{k!}t^{k}\right] \\
&=&\frac{2\left( -1\right) ^{n}\left( n+1\right) !}{\left( \frac{1}{2}%
\right) _{n}\,}\lim_{t\rightarrow 0}\frac{1}{t^{n+1}}\sum_{k=n+1}^{\infty }%
\frac{\left( \frac{1}{2}-n\right) _{k}}{k!}t^{k} \\
&=&\frac{2\left( -1\right) ^{n}\left( \frac{1}{2}-n\right) _{n+1}}{\left(
\frac{1}{2}\right) _{n}\,}=1,
\end{eqnarray*}%
where we have applied the property $\Gamma \left( z\right) \Gamma \left(
1-z\right) =\frac{\pi }{\sin \pi z}$ \cite[Eqn. 1.2.2]{Lebedev}.
\end{proof}

\begin{theorem}
For $n=0,1,2,\ldots $ and $\mathrm{Re}\left( s+x\right) >0$, we have%
\begin{eqnarray}
&&\int_{0}^{\infty }\frac{e^{-st}}{t^{3/2}}\gamma \left( n+1,xt\right) dt
\label{Int_1_resultado} \\
&=&-2\sqrt{\pi }n!\left[ \sqrt{s}-\sqrt{s+x}\sum_{k=0}^{n}\frac{\left( -%
\frac{1}{2}\right) _{k}}{k!}\left( \frac{x}{x+s}\right) ^{k}\right] ,  \notag
\end{eqnarray}%
where $\gamma \left( \nu ,z\right) $ denotes the lower incomplete gamma
function \cite[Chap. 45]{Atlas}.
\end{theorem}

\begin{proof}
Indeed, take $a_{1}=1+n$, $\alpha =\frac{1}{2}+n$ and $b_{1}=2+n$ in (\ref%
{Int_Andrews}), consider the result (\ref{2F1_resultado_(1)}), as well as
\cite[Eqn. 43:4:3]{Atlas}
\begin{equation*}
\Gamma \left( n+\frac{1}{2}\right) =\left( \frac{1}{2}\right) _{n}\sqrt{\pi }%
,
\end{equation*}%
to obtain%
\begin{eqnarray}
&&\int_{0}^{\infty }e^{-st}t^{n-1/2}\,_{1}F_{1}\left( \left.
\begin{array}{c}
1+n \\
2+n%
\end{array}%
\right\vert xt\right) dt  \label{Int_(1)} \\
&=&\frac{2\sqrt{\pi }\left( -1\right) ^{n}\left( n+1\right) !}{\,x^{n+1}}%
\left[ \sqrt{s}-\sqrt{s-x}\sum_{k=0}^{n}\frac{\left( -\frac{1}{2}\right) _{k}%
}{k!}\left( \frac{x}{x-s}\right) ^{k}\right] .  \notag
\end{eqnarray}%
However, according to \cite[Eqn. 7.11.1(13)]{Prudnikov3}, we have
\begin{equation}
_{1}F_{1}\left( \left.
\begin{array}{c}
n \\
1+n%
\end{array}%
\right\vert z\right) =\frac{\left( -1\right) ^{n}n!}{z^{n}}\left[
1-e^{-z}\sum_{k=0}^{n-1}\frac{\left( -1\right) ^{k}z^{k}}{k!}\right] ,
\label{1F1_Prudnikov}
\end{equation}%
and \cite[Eqns. 45:4:2\&26:12:2]{Atlas}, we have as well%
\begin{equation}
\Gamma \left( n,z\right) =\left( n-1\right) !e^{-z}e_{n-1}\left( z\right)
=\left( n-1\right) !e^{-z}\sum_{k=0}^{n-1}\frac{z^{k}}{k!},
\label{Gamma(n,z)}
\end{equation}%
where $\Gamma \left( \nu ,z\right) $ denotes the \textit{upper incomplete
gamma function} and $e_{n}\left( z\right) $ is the \textit{exponential
polynomial}. Therefore, from (\ref{1F1_Prudnikov})\ and (\ref{Gamma(n,z)}),
and taking into account that the\textit{\ lower incomplete gamma function}
satisfies \cite[Eqn. 45:0:1]{Atlas}
\begin{equation*}
\gamma \left( \nu ,z\right) =\Gamma \left( \nu \right) -\Gamma \left( \nu
,z\right) ,
\end{equation*}%
we conclude that%
\begin{equation}
_{1}F_{1}\left( \left.
\begin{array}{c}
n \\
1+n%
\end{array}%
\right\vert z\right) =n\left( -z\right) ^{-n}\gamma \left( n,-z\right) ,
\label{1F1_resultado_(1)}
\end{equation}%
hence, inserting (\ref{1F1_resultado_(1)})\ in (\ref{Int_(1)}), we arrive at
(\ref{Int_1_resultado}), as we wanted to prove.
\end{proof}

It is worth noting that we can obtain also (\ref{Int_1_resultado})\ from
\cite[Eqn. 2.10.3(2)]{Prudnikov2} and (\ref{2F1_resultado_(1)}).

\subsection{Second differentiation formula}

\begin{definition}[Regularized hypergeometric function]
\begin{equation}
_{p}\tilde{F}_{q}\left( \left.
\begin{array}{c}
a_{1},\ldots ,a_{p} \\
b_{1},\ldots ,b_{q}%
\end{array}%
\right\vert z\right) = \sum_{k=0}^{\infty }\frac{\left( a_{1}\right)
_{k}\cdots \left( a_{p}\right) _{k}}{\Gamma \left( b_{1} + k \right) \cdots
\Gamma \left( b_{q} + k \right)}\frac{z^{k}}{k!}.
\label{Hyper_Normalized_Generalized_def}
\end{equation}%
When $p\leq q + 1$ and $z$ is fixed and not a branch point, (\ref%
{Hyper_Normalized_Generalized_def})\ is an entire function of each of the
parameters $a_{1},\ldots ,a_{p},b_{1},\ldots ,b_{q}$ (see \cite[Eqn. 15.2.2]%
{DLMF}).
\end{definition}

\begin{theorem}
For $n=0,1,2,\ldots $ and $t\in
\mathbb{C}
\backslash \left\{ 1\right\} $,%
\begin{equation}
_{2}\tilde{F}_{1}\left( \left.
\begin{array}{c}
\frac{1}{2},1 \\
1-n%
\end{array}%
\right\vert t\right) =\frac{\left( \frac{1}{2}\right) _{n}}{\sqrt{1-t}}%
\left( \frac{t}{1-t}\right) ^{n}.  \label{2F1_resultado_(2)}
\end{equation}
\end{theorem}

\begin{proof}
In \cite[Eqn. 15.5.4]{DLMF}, we found the differentiation formula:%
\begin{eqnarray}
&&\frac{d^{n}}{dt^{n}}\left[ t^{c-1}\,_{2}F_{1}\left( \left.
\begin{array}{c}
a,b \\
c%
\end{array}%
\right\vert t\right) \right] =\left( c-n\right)
_{n}\,t^{c-n-1}\,_{2}F_{1}\left( \left.
\begin{array}{c}
a,b \\
c-n%
\end{array}%
\right\vert t\right) ,  \label{Dn_(2)} \\
&&n=0,1,2,\ldots  \notag
\end{eqnarray}%
thus taking $a=\frac{1}{2}$, $b=1$ and $c=2$ in (\ref{Dn_(2)})\ and
considering (\ref{2F1_resultado}), we have%
\begin{equation}
2\frac{d^{n+1}}{dt^{n+1}}\left[ 1-\sqrt{1-t}\right] =\frac{1}{\Gamma \left(
1-n\right) t^{n}}\,_{2}F_{1}\left( \left.
\begin{array}{c}
\frac{1}{2},1 \\
1-n%
\end{array}%
\right\vert t\right) .  \label{Eqn_(2)_1}
\end{equation}%
Apply (\ref{Dn(sqrt(1-t))})-(\ref{(x)_n+1}),\ and the definition of the
\textit{regularized hypergeometric function} given in (\ref{Hyper_Normalized_Generalized_def}) in order to rewrite (\ref%
{Eqn_(2)_1}) as (\ref{2F1_resultado_(2)}), as we wanted to prove.
\end{proof}

\begin{remark}
According to (\ref{Gauss_sum}), note that for $t=1$, both sides of (\ref%
{2F1_resultado_(2)})\ are divergent.
\end{remark}

\subsection{Third differentiation formula}

\begin{theorem}
For $n=0,1,2,\ldots $ and $t\in
\mathbb{C}
$, the following reduction formula holds true:\
\begin{eqnarray}
&&_{2}F_{1}\left( \left.
\begin{array}{c}
\frac{1}{2},1 \\
2+n%
\end{array}%
\right\vert t\right)  \label{2F1_resultado_(3)} \\
&=&\left\{
\begin{array}{ll}
\displaystyle%
\frac{2\left( n+1\right) !}{\left( \frac{3}{2}\right) _{n}\,\sqrt{1-t}}%
\left( \frac{t-1}{t}\right) ^{n+1}\left[ 1-\frac{1}{\sqrt{1-t}}\sum_{k=0}^{n}%
\frac{\left( \frac{1}{2}\right) _{k}}{k!}\left( \frac{t}{t-1}\right) ^{k}%
\right] , & t\neq 0,1, \\
\displaystyle%
\frac{2\left( n+1\right) }{2n+1}, & t=1, \\
1, & t=0.%
\end{array}%
\right.  \notag
\end{eqnarray}
\end{theorem}

\begin{proof}
For the case $t\neq 0,1$, set $a=\frac{1}{2}$, $b=1$ and $c=2$ in the
differentiation formula \cite[Eqn. 15.5.6]{DLMF},
\begin{eqnarray}
&&\frac{d^{n}}{dt^{n}}\left[ \left( 1-t\right) ^{a+b-c}\,_{2}F_{1}\left(
\left.
\begin{array}{c}
a,b \\
c%
\end{array}%
\right\vert t\right) \right]  \label{Dn_(3)} \\
&=&\frac{\left( c-a\right) _{n}\left( c-b\right) _{n}}{\left( c\right) _{n}}%
\left( 1-t\right) ^{a+b-c-n}\,_{2}F_{1}\left( \left.
\begin{array}{c}
a,b \\
c+n%
\end{array}%
\right\vert t\right) ,  \notag \\
&&n=0,1,2,\ldots  \notag
\end{eqnarray}%
and use the result (\ref{2F1_resultado}) to arrive at%
\begin{equation*}
2\frac{d^{n}}{dt^{n}}\left[ \frac{1}{t}\frac{1}{\sqrt{1-t}}-\frac{1}{t}%
\right] =\frac{\left( \frac{3}{2}\right) _{n}}{\left( n+1\right) \left(
1-t\right) ^{n+1/2}}\,_{2}F_{1}\left( \left.
\begin{array}{c}
\frac{1}{2},1 \\
2+n%
\end{array}%
\right\vert t\right) .
\end{equation*}%
Apply now Leibniz's differentiation formula (\ref{Leibniz_formula})\ and the
differentiation formulas (\ref{Dn(1/t)})\ and (\ref{Dn(1/sqrt)}). After some
algebra, we obtain (\ref{2F1_resultado_(3)}), as we wanted to prove.

For $t=1$, apply Gauss summation formula (\ref{Gauss_sum}), to obtain%
\begin{equation}
_{2}F_{1}\left( \left.
\begin{array}{c}
\frac{1}{2},1 \\
2+n%
\end{array}%
\right\vert 1\right) =\frac{2\left( n+1\right) }{2n+1},  \label{Gauss_2}
\end{equation}%
where (\ref{Gauss_2})\ only converges for $\mathrm{Re}\,\left( \frac{1}{2}%
+n\right) >0$, i.e. for $n=0,1,\ldots $, as we wanted to prove.

Finally, according to (\ref{Hyper_Generalized_def}), for $t=0$ we have%
\begin{equation*}
_{2}F_{1}\left( \left.
\begin{array}{c}
\frac{1}{2},1 \\
2+n%
\end{array}%
\right\vert 0\right) =1,
\end{equation*}%
as we wanted to prove.
\end{proof}

It is worth noting that we can provide other elementary representations for $%
_{2}F_{1}\left( 1/2,1;2+n;t\right) $, by using known formulas given in the
literature.

\begin{theorem}
For $n=0,1,2,\ldots $ and $t\in
\mathbb{C}
$, the following reduction formula holds true:\
\begin{eqnarray}
&&_{2}F_{1}\left( \left.
\begin{array}{c}
\frac{1}{2},1 \\
2+n%
\end{array}%
\right\vert t\right)   \label{2F1_resultado_(3)b} \\
&=&\left\{
\begin{array}{ll}
\begin{array}{l}
\displaystyle%
\frac{2\left( n+1\right) !}{\left( \frac{3}{2}\right) _{n}\,} \\
\displaystyle%
\left[ \frac{2}{\sqrt{1-t}}\left( \frac{t-1}{t}\right) ^{n+1}+\frac{1}{1-%
\sqrt{1-t}}\sum_{k=0}^{n}\frac{\left( n+1\right) _{k}}{k!\,2^{k+n}}\left( 1-%
\frac{1}{\sqrt{1-t}}\right) ^{k-n}\right] ,%
\end{array}
& t\neq 0,1, \\
\displaystyle%
\frac{2\left( n+1\right) }{2n+1}, & t=1, \\
1, & t=0.%
\end{array}%
\right.   \notag
\end{eqnarray}
\end{theorem}

\begin{proof}
We need to prove (\ref{2F1_resultado_(3)b})\ for $t\neq 0,1$. In \cite%
{Vidunas}, we found
\begin{equation*}
_{2}F_{1}\left( \left.
\begin{array}{c}
\frac{\alpha }{2},\frac{\alpha +1}{2} \\
\alpha +n+1%
\end{array}%
\right\vert t\right) =\left( \frac{1+\sqrt{1-t}}{2}\right) ^{-\alpha
}\,_{2}F_{1}\left( \left.
\begin{array}{c}
-n,\alpha \\
\alpha +n+1%
\end{array}%
\right\vert \frac{1-\sqrt{1-t}}{1+\sqrt{1-t}}\right) ,
\end{equation*}%
hence for $\alpha =1$, we obtain%
\begin{equation*}
_{2}F_{1}\left( \left.
\begin{array}{c}
\frac{1}{2},1 \\
2+n%
\end{array}%
\right\vert t\right) =\frac{2}{1+\sqrt{1-t}}\,_{2}F_{1}\left( \left.
\begin{array}{c}
-n,1 \\
2+n%
\end{array}%
\right\vert \frac{1-\sqrt{1-t}}{1+\sqrt{1-t}}\right) .
\end{equation*}

Now apply \cite[Eqn. 7.3.1(179)]{Prudnikov3}%
\begin{equation*}
_{2}F_{1}\left( \left.
\begin{array}{c}
-n,1 \\
m%
\end{array}%
\right\vert z\right) =\frac{-n!}{\left( m\right) _{n}\,z}\left( \frac{z-1}{z}%
\right) ^{m-2}\left[ \left( 1-z\right) ^{n+1}-\sum_{k=0}^{m-2}\frac{\left(
n+1\right) _{k}}{k!}\left( \frac{z}{z-1}\right) ^{k}\right] ,
\end{equation*}%
taking $m=n+2$ and $z=\frac{1-\sqrt{1-t}}{1+\sqrt{1-t}}$. Knowing that $%
\left( n+2\right) _{n}=\frac{2^{2n}}{n+1}\left( \frac{3}{2}\right) _{n}$,
after some algebra, we arrive at (\ref{2F1_resultado_(3)b}) for $t\neq 0,1$,
as we wanted to prove.
\end{proof}

\begin{theorem}
For $n=0,1,2,\ldots $ and $t\in
\mathbb{C}
$, the following reduction formula holds true:\
\begin{eqnarray}
&&_{2}F_{1}\left( \left.
\begin{array}{c}
\frac{1}{2},1 \\
2+n%
\end{array}%
\right\vert t\right)  \label{2F1_resultado(3)c} \\
&=&\left\{
\begin{array}{ll}
\displaystyle%
\frac{2\left( n+1\right) !}{\left( \frac{3}{2}\right) _{n}\,\left( -t\right)
^{n+1}}\left[ \left( 1-t\right) ^{n+1/2}+\sum_{k=0}^{n}\frac{\left( -n-\frac{%
1}{2}\right) _{k}}{k!}t^{k}\right] , & t\neq 0,1, \\
\displaystyle%
\frac{2\left( n+1\right) }{2n+1}, & t=1, \\
1, & t=0.%
\end{array}%
\right.  \notag
\end{eqnarray}
\end{theorem}

\begin{proof}
We need to prove (\ref{2F1_resultado(3)c})\ for $t\neq 0,1$. For this
purpose, apply (\ref{Formula_Prudnikov})\ taking $m=n+2$ and $b=\frac{1}{2}$
and use (\ref{(x)_n+1})\ for $x=\frac{1}{2}$.
\end{proof}

\begin{corollary}
For $n=1,2,\ldots $ and $t\in
\mathbb{C}
\backslash \left\{ 1\right\} $, we have
\begin{eqnarray}
&&P_{-n}^{-n}\left( \frac{1}{\sqrt{1-t}}\right)  \label{P-n-n_resultado} \\
&=&\left\{
\begin{array}{ll}
\displaystyle%
\frac{1}{2^{n}\left( \frac{1}{2}\right) _{n}}\left( \frac{t-1}{t}\right)
^{n/2}\left[ 1-\frac{1}{\sqrt{1-t}}\sum_{k=0}^{n-1}\frac{\left( \frac{1}{2}%
\right) _{k}}{k!}\left( \frac{t}{t-1}\right) ^{k}\right] , & t\neq 0, \\
0, & t=0.%
\end{array}%
\right.  \notag
\end{eqnarray}%
where $P_{\nu }^{\mu }\left( z\right) $ denotes the \textit{Legendre function%
} \cite[Chap. III]{Ederlyi}.
\end{corollary}

\begin{proof}
For $t\neq 0,1$, we found, in \cite[Eqn. 7.3.1(101)]{Prudnikov3},%
\begin{equation}
_{2}F_{1}\left( \left.
\begin{array}{c}
a,a+\frac{1}{2} \\
c%
\end{array}%
\right\vert t\right) =2^{c-1}\Gamma \left( c\right) \left( -t\right)
^{\left( 1-c\right) /2}\left( 1-t\right) ^{\left( c-1\right)
/2-a}\,P_{2a-c}^{1-c}\left( \frac{1}{\sqrt{1-t}}\right) .
\label{2F1_Legendre}
\end{equation}%
Therefore, taking $a=\frac{1}{2}$ and $c=2+n$ in (\ref{2F1_Legendre}),\ and
considering (\ref{2F1_resultado_(3)}) and (\ref{(x)_n+1}), we eventually
arrive at (\ref{P-n-n_resultado}), as we wanted to prove.

For $t=0$, apply the hypergeometric representation of the Legendre function
\cite[Eqn. 14.3.15]{DLMF}
\begin{equation*}
P_{\nu }^{-\mu }\left( x\right) =2^{-\mu }\left( x^{2}-1\right) ^{\mu
/2}\,_{2}F_{1}\left( \left.
\begin{array}{c}
\mu -\nu ,\mu +\nu +1 \\
\mu +1%
\end{array}%
\right\vert \frac{1-x}{2}\right) ,
\end{equation*}%
to conclude that $P_{-n}^{-n}\left( 1\right) =0$, for $n=1,2,\ldots $
\end{proof}

\begin{corollary}
For $n=1,2,\ldots $ and $x,p\in
\mathbb{C}
$, we have%
\begin{eqnarray}
&&\int_{0}^{\infty }\frac{e^{-pt}}{t^{1/2+n}}\gamma \left( n,xt\right) dt
\label{Int_2_resultado} \\
&=&\left\{
\begin{array}{ll}
\displaystyle%
\frac{-\sqrt{\pi }\left( n-1\right) !\left( -p\right) ^{n-1}}{\left( \frac{1%
}{2}\right) _{n}}\left[ \sqrt{p}-\sqrt{p+x}\sum_{k=0}^{n-1}\frac{\left(
\frac{1}{2}\right) _{k}}{k!}\left( -\frac{x}{p}\right) ^{k}\right] , & p\neq
0, \\
\displaystyle%
\frac{2\sqrt{\pi }x^{n-1/2}}{2n-1}, & p=0.%
\end{array}%
\right.  \notag
\end{eqnarray}
\end{corollary}

\begin{proof}
For $p\neq 0$, take $a_{1}=1$, $\alpha =\frac{1}{2}$ and $b_{1}=1+n$ in (\ref%
{Int_Andrews}), consider the reduction formula of the Kummer function \cite[%
Eqn. 7.11.1(14)]{Prudnikov3}, i.e.%
\begin{equation*}
_{1}F_{1}\left( \left.
\begin{array}{c}
1 \\
1+n%
\end{array}%
\right\vert xt\right) =\frac{n\,e^{xt}}{\left( xt\right) ^{n}}\gamma \left(
n,xt\right) ,
\end{equation*}%
and apply the result given in (\ref{2F1_resultado_(3)}) and the property (%
\ref{(x)_n+1}), to arrive after some algebra at (\ref{Int_2_resultado}), as
we wanted to prove.

For $p=0$, rewrite the result obtained above as
\begin{eqnarray*}
&&\int_{0}^{\infty }\frac{e^{-pt}}{t^{1/2+n}}\gamma \left( n,xt\right) dt \\
&=&\frac{\sqrt{\pi }\left( n-1\right) !\left( -1\right) ^{n}}{\left( \frac{1%
}{2}\right) _{n}} \\
&&\left\{ p^{n-1/2}-\sqrt{p+x}\left[ \sum_{k=0}^{n-2}\frac{\left( \frac{1}{2}%
\right) _{k}}{k!}\left( -x\right) ^{k}p^{n-1-k}+\frac{\left( \frac{1}{2}%
\right) _{n-1}}{\left( n-1\right) !}\left( -x\right) ^{k}\right] \right\} ,
\end{eqnarray*}%
and take $p=0$, to obtain the desired result.
\end{proof}

It is worth noting that we can obtain (\ref{Int_2_resultado})\ from \cite[%
Eqn. 2.10.3(2)]{Prudnikov2} and (\ref{2F1_resultado_(3)}).

\subsection{Fourth differentiation formula}

\begin{theorem}
For $n=1,2,\ldots $ and $t\in
\mathbb{C}
$, we have%
\begin{equation}
_{2}\tilde{F}_{1}\left( \left.
\begin{array}{c}
\frac{1}{2}-n,1-n \\
2-n%
\end{array}%
\right\vert t\right) =\left\{
\begin{array}{ll}
\displaystyle%
2\left( \frac{1}{2}\right) _{n}\,t^{n-1}, & n\geq 1, \\
1, & n=1.%
\end{array}%
\right. ,  \label{2F1_regularized_resultado}
\end{equation}
\end{theorem}

\begin{proof}
Set $a=\frac{1}{2}$, $b=1$ and $c=2$ in the differentiation formula \cite[%
Eqn. 15.5.9]{DLMF},%
\begin{eqnarray}
&&\frac{d^{n}}{dt^{n}}\left[ t^{c-1}\left( 1-t\right)
^{a+b-c}\,_{2}F_{1}\left( \left.
\begin{array}{c}
a,b \\
c%
\end{array}%
\right\vert t\right) \right]  \label{Dn_(4)} \\
&=&\left( c-n\right) _{n}\,t^{c-n-1}\left( 1-t\right)
^{a+b-c-n}\,_{2}F_{1}\left( \left.
\begin{array}{c}
a-n,b-n \\
c-n%
\end{array}%
\right\vert t\right) ,  \notag \\
&&n=0,1,2,\ldots  \notag
\end{eqnarray}%
and apply the result given in (\ref{2F1_resultado}), to obtain%
\begin{equation*}
2\frac{d^{n}}{dt^{n}}\left( \frac{1}{\sqrt{1-t}}-1\right) =\frac{%
t^{1-n}\left( 1-t\right) ^{-1/2-n}}{\Gamma \left( 2-n\right) }%
\,_{2}F_{1}\left( \left.
\begin{array}{c}
\frac{1}{2}-n,1-n \\
2-n%
\end{array}%
\right\vert t\right) .
\end{equation*}%
According to (\ref{Dn(1/sqrt)}) for $n\geq 1$\ and the definition of the
regularized generalized hypergeometric function given in (\ref{Hyper_Normalized_Generalized_def}), we finally get (\ref%
{2F1_regularized_resultado}).

For $t=0$ and $n=1$, we obtain a indeterminate expression. However,
according to (\ref{Hyper_Generalized_def}), we have that
\begin{equation*}
\,_{2}F_{1}\left( \left.
\begin{array}{c}
a,0 \\
b%
\end{array}%
\right\vert t\right) =1,
\end{equation*}%
thus we obtain the desired result for $n=1$.
\end{proof}

\begin{corollary}
The following identity holds true for $n=1,2,\ldots $ and $t\in
\mathbb{C}
$,
\begin{equation}
P_{n}^{n-1}\left( t\right) =-\left( -2\right) ^{n}\left( \frac{1}{2}\right)
_{n}\,t\left( 1-t^{2}\right) ^{\left( n-1\right) /2}.  \label{P_n_n-1}
\end{equation}
\end{corollary}

\begin{proof}
Set $a=\frac{1}{2}-n$ and $c=2-n$ in (\ref{2F1_Legendre}), and take into
account (\ref{2F1_regularized_resultado}), to obtain%
\begin{equation*}
P_{-n-1}^{n-1}\left( \frac{1}{\sqrt{1-t}}\right) =-\frac{\left( -2\right)
^{n}\left( \frac{1}{2}\right) _{n}}{\sqrt{1-t}}\left( \frac{t}{t-1}\right)
^{\left( n-1\right) /2},
\end{equation*}%
which, according to the property \cite[Eqn. 3.3.1(1)]{Ederlyi}:
\begin{equation}
P_{-\nu -1}^{\mu }\left( z\right) =P_{\nu }^{\mu }\left( z\right) ,
\label{Pnumu}
\end{equation}%
is equivalent to (\ref{P_n_n-1}).
\end{proof}

\section{Case $n=\ 3$}

In this case, (\ref{Main_Eqn})\ becomes
\begin{equation}
x^{3}-x+t=0.  \label{Eqn_cubic}
\end{equation}

In order to solve (\ref{Eqn_cubic}), we apply the solution of the cubic
equation given in Appendix $A$, considering in (\ref{Cubic_depressed})\ the
negative sign `$-$', $m=\frac{1}{3}$ and $n=\frac{t}{2}$, i.e.
\begin{equation}
x_{3}\left( t\right) =\frac{1}{\sqrt{3}}\left\{
\begin{array}{ll}
\cosh \left( \frac{1}{3}\cosh ^{-1}\sqrt{z}\right) -i\sqrt{3}\sinh \left(
\frac{1}{3}\cosh ^{-1}\sqrt{z}\right) , & z\geq 1, \\
\cos \left( \frac{1}{3}\cos ^{-1}\sqrt{z}\right) -\sqrt{3}\sin \left( \frac{1%
}{3}\cos ^{-1}\sqrt{z}\right) , & z\leq 1.%
\end{array}%
\right.  \label{x3(t)_b}
\end{equation}
where $z=3\left( \frac{3t}{2}\right) ^{2}$. Therefore, from (\ref{x3(t)_Hyp}%
) and (\ref{x3(t)_b})\ we have%
\begin{eqnarray}
&&_{2}F_{1}\left( \left.
\begin{array}{c}
\frac{1}{3},\frac{2}{3} \\
\frac{3}{2}%
\end{array}%
\right\vert z\right)  \label{2F1_(3)} \\
&=&\frac{3}{2\sqrt{z}}\left\{
\begin{array}{ll}
\cosh \left( \frac{1}{3}\cosh ^{-1}\sqrt{z}\right) -i\sqrt{3}\sinh \left(
\frac{1}{3}\cosh ^{-1}\sqrt{z}\right) , & z\geq 1, \\
\cos \left( \frac{1}{3}\cos ^{-1}\sqrt{z}\right) -\sqrt{3}\sin \left( \frac{1%
}{3}\cos ^{-1}\sqrt{z}\right) , & z\leq 1.%
\end{array}%
\right.  \notag
\end{eqnarray}

Note that we can simplify (\ref{2F1_(3)}) considering that
\begin{eqnarray*}
\frac{3}{\sqrt{z}}\sin \left( \frac{1}{3}\sin ^{-1}\sqrt{z}\right) &=&\frac{3%
}{\sqrt{z}}\sin \left( \frac{\pi /2-\cos ^{-1}\sqrt{z}}{3}\right) \\
&=&\frac{3}{2\sqrt{z}}\left\{ \cos \left( \frac{\cos ^{-1}\sqrt{z}}{3}%
\right) -\sqrt{3}\sin \left( \frac{\cos ^{-1}\sqrt{z}}{3}\right) \right\} .
\end{eqnarray*}

Since
\begin{equation*}
\cos ^{-1}x=\left\{
\begin{array}{ll}
i\cosh ^{-1}x, & x\geq 1, \\
-i\cosh ^{-1}x, & x\leq 1,%
\end{array}%
\right.
\end{equation*}%
and $\cos \left( ix\right) =\cosh x$, and $\sin \left( ix\right) =i\sinh x$,
we conclude that $\forall z\in
\mathbb{C}
$,%
\begin{equation}
_{2}F_{1}\left( \left.
\begin{array}{c}
\frac{1}{3},\frac{2}{3} \\
\frac{3}{2}%
\end{array}%
\right\vert z\right) =\frac{3}{\sqrt{z}}\sin \left( \frac{1}{3}\sin ^{-1}%
\sqrt{z}\right) .  \label{2F1_(3)_resultado}
\end{equation}

The result given (\ref{2F1_(3)_resultado}) can be obtained from \cite[Eqn.
2.8(12)]{Ederlyi}:%
\begin{equation*}
_{2}F_{1}\left( \left.
\begin{array}{c}
\frac{1+a}{2},\frac{1-a}{3} \\
\frac{3}{2}%
\end{array}%
\right\vert \sin ^{2}z\right) =\frac{\sin a\,z}{a\sin z},
\end{equation*}%
taking $a=\frac{1}{3}$. Nonetheless, by differentiation, we obtain from (\ref%
{2F1_(3)_resultado})\ the following interesting identity.

\begin{theorem}
For $n=1,2,\ldots $ and $z\in
\mathbb{C}
\backslash \left\{ 0,1\right\} $, we have:%
\begin{eqnarray}
&&_{2}\tilde{F}_{1}\left( \left.
\begin{array}{c}
\frac{1}{3},\frac{2}{3} \\
\frac{3}{2}-n%
\end{array}%
\right\vert z\right)  \label{2F1_Bell_resultado} \\
&=&\frac{6\,z^{n-1/2}}{\sqrt{\pi }}\frac{d^{n}}{dz^{n}}\left[ \sin \left(
\frac{1}{3}\sin ^{-1}\sqrt{z}\right) \right]  \label{Bell_2} \\
&=&\frac{6\,z^{n-1/2}}{\sqrt{\pi }}\sum_{k=1}^{n}\sin \left( \frac{\sin ^{-1}%
\sqrt{z}}{3}+\frac{\pi k}{2}\right) B_{n,k}\left( h_{1}\left( z\right)
,\ldots ,h_{n-k+1}\left( z\right) \right) ,  \label{Bell_3}
\end{eqnarray}%
where $B_{n,k}\left( x_{1},\ldots ,x_{n-k+1}\right) $ denotes the \textit{%
Bell polynomial} \cite[p. 133]{Comtet}. Also, we have defined
\begin{equation*}
h_{s}\left( z\right) =\frac{\left( -i\,\right) ^{s-1}\left( s-1\right) !}{6%
\left[ z\left( 1-z\right) \right] ^{s/2}}\ P_{s-1}\left( \frac{1-2z}{2\sqrt{%
z\left( z-1\right) }}\right) ,
\end{equation*}%
where $P_{n}\left( x\right) $ is a \textit{Legendre polynomial.}
\end{theorem}

\begin{proof}
Set $a=\frac{1}{3}$, $b=\frac{2}{3}$, and $c=\frac{3}{2}$ in (\ref{Dn_(2)})\
to obtain%
\begin{eqnarray}
&&\frac{1}{3}\frac{d^{n}}{dz^{n}}\left[ \sqrt{z}\,_{2}F_{1}\left( \left.
\begin{array}{c}
\frac{1}{3},\frac{2}{3} \\
\frac{3}{2}%
\end{array}%
\right\vert z\right) \right]   \label{Dn(2F1)_(3)} \\
&=&\frac{\sqrt{\pi }\,z^{1/2-n}}{6\,}\,_{2}\tilde{F}_{1}\left( \left.
\begin{array}{c}
\frac{1}{3},\frac{2}{3} \\
\frac{3}{2}-n%
\end{array}%
\right\vert z\right) ,  \notag
\end{eqnarray}%
and substitute (\ref{2F1_(3)_resultado})\ in (\ref{Dn(2F1)_(3)}), to get%
\begin{equation}
_{2}\tilde{F}_{1}\left( \left.
\begin{array}{c}
\frac{1}{3},\frac{2}{3} \\
\frac{3}{2}-n%
\end{array}%
\right\vert z\right) =\frac{6\,z^{n-1/2}}{\sqrt{\pi }}\frac{d^{n}}{dz^{n}}%
\left[ \sin \left( \frac{1}{3}\sin ^{-1}\sqrt{z}\right) \right] .
\label{Dn(sin)}
\end{equation}%
In order to calculate the $n$-th derivative given in (\ref{Dn(sin)}), we
apply Fa\`{a} di Bruno's formula \cite[p. 137]{Comtet}:%
\begin{equation}
\frac{d^{n}}{dz^{n}}f\left[ g\left( z\right) \right] =\sum_{k=1}^{n}f^{%
\left( k\right) }\left[ g\left( z\right) \right] \ B_{n,k}\left( g^{\prime
}\left( z\right) ,g^{\prime \prime }\left( z\right) ,\ldots ,g^{\left(
n-k+1\right) }\left( z\right) \right) ,  \label{Faa_DiBruno}
\end{equation}%
Set $f\left( z\right) =\sin z$ and $g\left( z\right) =\frac{1}{3}\sin ^{-1}%
\sqrt{z}$ in (\ref{Faa_DiBruno})\ and take into account the differentiation
formula \cite[Eqn. 1.1.7(7)]{Brychov}:%
\begin{eqnarray*}
&&\frac{d^{n}}{dz^{n}}\sin ^{-1}\left( a\sqrt{z}\right) =\frac{\left(
-i\right) ^{n-1}}{2}\left( n-1\right) !a^{n}\left( z-a^{2}z^{2}\right)
^{-n/2}P_{n-1}\left( \frac{1-2a^{2}z}{2a\sqrt{a^{2}z^{2}-z}}\right) , \\
&&n\geq 1,
\end{eqnarray*}%
to arrive at (\ref{2F1_Bell_resultado}), as we wanted to prove.
\end{proof}

\begin{remark}
On \ the one hand, according to Gauss summation formula (\ref{Gauss_sum}),
the reguralized hypergeometric function given in (\ref{2F1_Bell_resultado})\
is divergent for $z=1$ and $n=1,2,\ldots $ On the other hand, for $z=0$, (%
\ref{Bell_2})\ and (\ref{Bell_3})\ yield indeterminate expressions. However,
according to (\ref{Hyper_z=1}), $_{2}F_{1}\left( 1/2,2/3;3/2-n;0\right) =1$
for $n=1,2,\ldots $
\end{remark}

Next, we provide the elementary representations of (\ref{2F1_Bell_resultado}%
)\ for $n=1,2$:%
\begin{equation*}
_{2}F_{1}\left( \left.
\begin{array}{c}
\frac{1}{3},\frac{2}{3} \\
\frac{1}{2}%
\end{array}%
\right\vert z\right) =\frac{\cos \left( \frac{1}{3}\sin ^{-1}\sqrt{z}\right)
}{\sqrt{1-z}},
\end{equation*}%
and
\begin{eqnarray*}
&&_{2}F_{1}\left( \left.
\begin{array}{c}
\frac{1}{3},\frac{2}{3} \\
-\frac{1}{2}%
\end{array}%
\right\vert z\right) \\
&=&\frac{\left( 3-6z\right) \cos \left( \frac{1}{3}\sin ^{-1}\sqrt{z}\right)
+\sqrt{-z\left( z-1\right) }\sin \left( \frac{1}{3}\sin ^{-1}\sqrt{z}\right)
}{3\left( 1-z\right) ^{3/2}}.
\end{eqnarray*}

\begin{theorem}
For $\mathrm{Re}\,\left( 2p-x\right) >0$, the following definite integral
holds true:%
\begin{eqnarray}
&&\int_{0}^{\infty }\frac{e^{-pt}}{t^{5/6}}D_{1/3}\left( -\sqrt{2xt}\right)
dt  \label{Int_3_resultado} \\
&=&\frac{2\,\Gamma \left( \frac{1}{3}\right) }{\left( 2p+x\right) ^{1/6}}%
\left[ \cos \left( \frac{1}{3}\cos ^{-1}\sqrt{\frac{2x}{2p+x}}\right) -\sin
\left( \frac{1}{3}\sin ^{-1}\sqrt{\frac{2x}{2p+x}}\right) \right] ,  \notag
\end{eqnarray}%
where $D_{\nu }\left( z\right) $ denotes the \textit{parabolic cylinder
function} \cite[Chap. VIII]{Ederlyi2}.
\end{theorem}

\begin{proof}
Set $a_{1}=\frac{1}{3}$, $b_{1}=\frac{3}{2}$, and $\alpha =\frac{2}{3}$ in (%
\ref{Int_Andrews})\footnote{%
It is worth noting that the other choice, i.e. $a_{1}=\frac{2}{3}$ and $%
\alpha =\frac{1}{3}$, leads to non-convergent integrals.}, taking into
account (\ref{2F1_(3)_resultado}), to obtain%
\begin{equation}
\int_{0}^{\infty }\frac{e^{-st}}{t^{1/3}}\,_{1}F_{1}\left( \left.
\begin{array}{c}
\frac{1}{3} \\
\frac{3}{2}%
\end{array}%
\right\vert xt\right) dt=\frac{3\Gamma \left( \frac{2}{3}\right) }{s^{1/6}%
\sqrt{x}}\sin \left( \frac{1}{3}\sin ^{-1}\sqrt{\frac{x}{s}}\right) .
\label{Int_(3)_1}
\end{equation}%
Apply now\ the following formula with $a=\frac{1}{3}$ \cite[Eqn. 7.11.1(10)]%
{Prudnikov3}:%
\begin{equation*}
_{1}F_{1}\left( \left.
\begin{array}{c}
a \\
\frac{3}{2}%
\end{array}%
\right\vert z\right) =\frac{2^{a-5/2}}{\sqrt{\pi \,z}}\Gamma \left( a-\frac{1%
}{2}\right) e^{z/2}\left[ D_{1-2a}\left( -\sqrt{2z}\right) -D_{1-2a}\left(
\sqrt{2z}\right) \right] ,
\end{equation*}%
hence the RHS\ of (\ref{Int_(3)_1})\ becomes:%
\begin{eqnarray}
&&\int_{0}^{\infty }\frac{e^{-st}}{t^{1/3}}\,_{1}F_{1}\left( \left.
\begin{array}{c}
\frac{1}{3} \\
\frac{3}{2}%
\end{array}%
\right\vert xt\right) dt=\frac{2^{-13/6}}{\sqrt{\pi \,x}}\Gamma \left( \frac{%
-1}{6}\right)  \label{Int_(3)_2} \\
&&\qquad \qquad \left[ \int_{0}^{\infty }\frac{e^{-\left( s-x/2\right) t}}{%
t^{5/6}}D_{1/3}\left( -\sqrt{2xt}\right) dt-\int_{0}^{\infty }\frac{%
e^{-\left( s-x/2\right) t}}{t^{5/6}}D_{1/3}\left( \sqrt{2xt}\right) dt\right]
.  \notag
\end{eqnarray}%
Consider now the definite integral \cite[Eqn. 8.3(11)]{Ederlyi2}:%
\begin{eqnarray*}
&&\int_{0}^{\infty }\frac{e^{-zt}}{t^{1-\beta /2}}D_{-\nu }\left( 2\sqrt{kt}%
\right) dt=\frac{2^{1-\beta -\nu /2}\sqrt{\pi }\Gamma \left( \beta \right) }{%
\Gamma \left( \frac{\nu +\beta +1}{2}\right) \left( z+k\right) ^{\beta /2}}%
\,_{2}F_{1}\left( \left.
\begin{array}{c}
\frac{\nu }{2},\frac{\beta }{2} \\
\frac{\nu +\beta +1}{2}%
\end{array}%
\right\vert \frac{z-k}{z+k}\right) , \\
&&\mathrm{Re}\,\beta >0,\mathrm{Re}\,z/k>0,
\end{eqnarray*}%
and the reduction formula \cite[Eqn. 7.3.1(83)]{Prudnikov3}:%
\begin{equation*}
_{2}F_{1}\left( \left.
\begin{array}{c}
a,-a \\
\frac{1}{2}%
\end{array}%
\right\vert z\right) =\cos \left( 2a\sin ^{-1}\sqrt{z}\right) ,
\end{equation*}%
to arrive at%
\begin{equation}
\int_{0}^{\infty }\frac{e^{-\left( s-x/2\right) t}}{t^{5/6}}D_{1/3}\left(
\sqrt{2xt}\right) dt=\frac{2^{5/6}\Gamma \left( \frac{1}{3}\right) }{s^{1/6}}%
\cos \left( \frac{1}{3}\cos ^{-1}\sqrt{\frac{x}{s}}\right) .
\label{Int_(3)_3}
\end{equation}%
Therefore, taking into account (\ref{Int_(3)_1})-(\ref{Int_(3)_3}), as well
as \cite[Eqns. 1.2.1\&3]{Lebedev}:
\begin{equation*}
\frac{\Gamma \left( \frac{2}{3}\right) }{\Gamma \left( -\frac{1}{6}\right)
\Gamma \left( \frac{1}{3}\right) }=\frac{-1}{6\times 2^{1/3}\sqrt{\pi }},
\end{equation*}%
after some algebra, we conclude (\ref{Int_3_resultado}), as we wanted to
prove.
\end{proof}

\section{Case $n=\ 4$}

In this case, (\ref{Main_Eqn})\ becomes
\begin{equation}
x^{4}-x+t=0.  \label{Eqn_Quartic}
\end{equation}

To solve (\ref{Eqn_Quartic}), we consider $p=0$, $q=-1$ and $r=t$ in the
solution of the quartic equation given in Appendix $B$, i.e. (\ref%
{Quartic_depressed}). Thereby, (\ref{gamma_def}) and (\ref{beta_def}) become
\begin{eqnarray}
\gamma &=&\frac{1}{2}\left( \alpha ^{2}+\frac{1}{\alpha }\right) ,
\label{gamma_resultado} \\
\beta &=&\frac{t}{\gamma }.  \label{beta_rresultado}
\end{eqnarray}

Therefore, setting $\xi =\alpha ^{2}$, the resolvent cubic (\ref%
{Resolvent_cubic_def})\ is
\begin{equation*}
\xi ^{3}-4t\xi -1=0,
\end{equation*}%
which can be solved taking in (\ref{Cubic_depressed})\ the `$-$' sign, $m=%
\frac{4t}{3}$ and $n=-\frac{1}{2}$. Thereby, according to (\ref{x_cosh}) and
(\ref{x_cos}), and defining $z=4\left( \frac{4t}{3}\right) ^{3}$, we arrive
at
\begin{equation}
\xi \left( z\right) =\left\{
\begin{array}{ll}
-2^{2/3}z^{1/6}\cosh \left( \frac{1}{3}\cosh ^{-1}\left( \frac{-1}{\sqrt{z}}%
\right) \right) , & z\leq 1, \\
-2^{2/3}z^{1/6}\cos \left( \frac{1}{3}\cos ^{-1}\left( \frac{-1}{\sqrt{z}}%
\right) \right) , & z\geq 1.%
\end{array}%
\right.  \label{xi_def}
\end{equation}

Note that both branches in (\ref{xi_def})\ are equivalent, if we consider $%
z\in
\mathbb{C}
$, thus let us define the following function:

\begin{definition}
\begin{equation}
g\left( z\right) =-z^{1/6}\cosh \left( \frac{1}{3}\cosh ^{-1}\left( \frac{-1%
}{\sqrt{z}}\right) \right) .  \label{g(z)_def}
\end{equation}
\end{definition}

By inspection, the solution of (\ref{Main_Eqn})\ for $n=4$\ corresponding to
(\ref{x4(t)_Hyp}) is just the solution $x_{1}$ in (\ref{x_1,2_quartic}), i.e.%
\begin{equation}
x_{1}=\frac{1}{2}\left( -\alpha + \sqrt{\alpha ^{2}-4\beta }\right) .
\label{x1_quartic}
\end{equation}

Therefore, from (\ref{x4(t)_Hyp}) on the one hand,\ and from (\ref%
{gamma_resultado})-(\ref{x1_quartic})\ on the other hand, we finally obtain:

\begin{theorem}
For $z\in
\mathbb{C}
$, we have%
\begin{equation}
_{3}F_{2}\left( \left.
\begin{array}{c}
\frac{1}{4},\frac{1}{2},\frac{3}{4} \\
\frac{2}{3},\frac{4}{3}%
\end{array}%
\right\vert z\right) =\frac{4}{3}z^{-1/3}\left[ \sqrt{g\left( z\right) +%
\frac{3z^{1/3}\sqrt{g\left( z\right) }}{1-2\left[ g\left( z\right) \right]
^{3/2}}}-\sqrt{g\left( z\right) }\right] .  \label{3F2_resultado}
\end{equation}
\end{theorem}

\begin{remark}
It is worth noting that the numerical evaluation of the LHS of (\ref%
{3F2_resultado}) seems to fail for $z=1$, since this point is a branch
point. However, taking $a=\frac{1}{4}$, $c=\frac{1}{2}$ and $d=\frac{1}{3}$
in Whipple's sum (\ref{Whipple_sum}), we obtain:%
\begin{equation*}
_{3}F_{2}\left( \left.
\begin{array}{c}
\frac{1}{4},\frac{1}{2},\frac{3}{4} \\
\frac{2}{3},\frac{4}{3}%
\end{array}%
\right\vert 1\right) =\frac{\pi \,\Gamma \left( \frac{2}{3}\right) \Gamma
\left( \frac{4}{3}\right) }{\Gamma \left( \frac{11}{24}\right) \Gamma \left(
\frac{17}{24}\right) \Gamma \left( \frac{19}{24}\right) \Gamma \left( \frac{%
25}{24}\right) }=\frac{4}{3},
\end{equation*}%
which is the result that we obtain on the RHS of (\ref{3F2_resultado}).
\end{remark}

\begin{corollary}
For $\left\vert z\right\vert <1$, we have%
\begin{equation}
\,_{3}F_{2}\left( \left.
\begin{array}{c}
\frac{1}{2},\frac{5}{6},\frac{1}{6} \\
\frac{2}{3},\frac{4}{3}%
\end{array}%
\right\vert z\right) =\frac{1}{\sqrt{1-z}}H\left( \frac{-4z}{\left(
1-z\right) ^{2}}\right) ,  \label{3F2_resultado_b}
\end{equation}%
where%
\begin{equation*}
H\left( t\right) =\frac{4}{3}t^{-1/3}\left[ \sqrt{g\left( t\right) +\frac{%
3z^{1/3}\sqrt{g\left( t\right) }}{1-2\left[ g\left( t\right) \right] ^{3/2}}}%
-\sqrt{g\left( t\right) }\right] .
\end{equation*}
\end{corollary}

\begin{proof}
Take $\alpha =\frac{1}{4}$, $\lambda =\frac{1}{3},$ and $\mu =-\frac{1}{3}$
in the the quadratic transformation \cite{Kato}:%
\begin{eqnarray*}
&&_{3}F_{2}\left( \left.
\begin{array}{c}
2\alpha ,2\alpha +\lambda ,2\alpha +\mu \\
1-\lambda ,1-\mu%
\end{array}%
\right\vert x\right) \\
&=&\left( 1-x\right) ^{-2\alpha }\,_{3}F_{2}\left( \left.
\begin{array}{c}
\alpha ,\alpha +\frac{1}{2},1-2\alpha -\lambda -\mu \\
1-\lambda ,1-\mu%
\end{array}%
\right\vert \frac{-4x}{\left( 1-x\right) ^{2}}\right) ,
\end{eqnarray*}%
to obtain,
\begin{equation}
_{3}F_{2}\left( \left.
\begin{array}{c}
\frac{1}{2},\frac{5}{6},\frac{1}{6} \\
\frac{2}{3},\frac{4}{3}%
\end{array}%
\right\vert x\right) =\frac{1}{\sqrt{1-x}}\,_{3}F_{2}\left( \left.
\begin{array}{c}
\frac{1}{4},\frac{3}{4},\frac{1}{2} \\
\frac{2}{3},\frac{4}{3}%
\end{array}%
\right\vert \frac{-4x}{\left( 1-x\right) ^{2}}\right) .  \label{Kato}
\end{equation}%
From (\ref{3F2_resultado})\ and (\ref{Kato}), we arrive at (\ref%
{3F2_resultado_b}), as we wanted to prove.
\end{proof}

\begin{remark}
We can calculate the LHS\ of (\ref{3F2_resultado_b}) for the branch point $%
z=1$ taking $a=\frac{1}{6}$, $c=\frac{1}{2}$, and $d=\frac{2}{3}$, resulting
in%
\begin{equation*}
\,_{3}F_{2}\left( \left.
\begin{array}{c}
\frac{1}{2},\frac{5}{6},\frac{1}{6} \\
\frac{2}{3},\frac{4}{3}%
\end{array}%
\right\vert 1\right) =\frac{\pi \,\Gamma \left( \frac{2}{3}\right) \Gamma
\left( \frac{4}{3}\right) }{\Gamma \left( \frac{5}{12}\right) \Gamma
^{2}\left( \frac{3}{4}\right) \Gamma \left( \frac{13}{12}\right) }\approx
1.24081.
\end{equation*}
\end{remark}

\begin{corollary}
From the result (\ref{3F2_resultado}), we obtain the following identity
involving the product of two Legendre functions:%
\begin{eqnarray}
&&P_{-1/6}^{1/3}\left( \sqrt{\frac{2}{1+\sqrt{1-z}}}\right) \
P_{-1/6}^{-1/3}\left( \sqrt{\frac{2}{1+\sqrt{1-z}}}\right)  \label{Product_P}
\\
&=&\frac{\sqrt{6\left( 1+\sqrt{1-z}\right) }}{\pi \ z^{1/3}}\left[ \sqrt{%
g\left( z\right) +\frac{3z^{1/3}\sqrt{g\left( z\right) }}{1-2\left[ g\left(
z\right) \right] ^{3/2}}}-\sqrt{g\left( z\right) }\right] .  \notag
\end{eqnarray}
\end{corollary}

\begin{proof}
We found in the literature \cite[Eqn. 7.4.1(10)]{Prudnikov3}:%
\begin{eqnarray*}
&&_{3}F_{2}\left( \left.
\begin{array}{c}
a,1-a,\frac{1}{2} \\
b,2-b%
\end{array}%
\right\vert z\right) \\
&=&\,_{2}F_{1}\left( \left.
\begin{array}{c}
a,1-a \\
2-b%
\end{array}%
\right\vert \frac{1-\sqrt{1-z}}{2}\right) \,_{2}F_{1}\left( \left.
\begin{array}{c}
a,1-a \\
b%
\end{array}%
\right\vert \frac{1-\sqrt{1-z}}{2}\right) ,
\end{eqnarray*}%
thus, taking $a=\frac{1}{4}$ and $b=\frac{3}{3}$, we have%
\begin{eqnarray}
&&_{3}F_{2}\left( \left.
\begin{array}{c}
\frac{1}{4},\frac{1}{2},\frac{3}{4} \\
\frac{2}{3},\frac{4}{3}%
\end{array}%
\right\vert z\right)  \label{3F2_a} \\
&=&\,_{2}F_{1}\left( \left.
\begin{array}{c}
\frac{1}{4},\frac{3}{4} \\
\frac{4}{3}%
\end{array}%
\right\vert \frac{1-\sqrt{1-z}}{2}\right) \,_{2}F_{1}\left( \left.
\begin{array}{c}
\frac{1}{4},\frac{3}{4} \\
\frac{2}{3}%
\end{array}%
\right\vert \frac{1-\sqrt{1-z}}{2}\right) .  \notag
\end{eqnarray}%
Also, setting $a=\frac{1}{4}$ and $c=\frac{2}{3},\frac{4}{3}$ in (\ref%
{2F1_Legendre}), we have%
\begin{eqnarray}
_{2}F_{1}\left( \left.
\begin{array}{c}
\frac{1}{4},\frac{3}{4} \\
\frac{2}{3}%
\end{array}%
\right\vert z\right) &=&2^{-1/3}\Gamma \left( \frac{2}{3}\right)
z^{1/6}\left( 1-z\right) ^{-5/12}P_{-1/6}^{1/3}\left( \frac{1}{\sqrt{1-z}}%
\right) ,  \label{2F1_2/3} \\
_{2}F_{1}\left( \left.
\begin{array}{c}
\frac{1}{4},\frac{3}{4} \\
\frac{4}{3}%
\end{array}%
\right\vert z\right) &=&2^{1/3}\Gamma \left( \frac{4}{3}\right)
z^{-1/6}\left( 1-z\right) ^{-1/12}P_{-5/6}^{-1/3}\left( \frac{1}{\sqrt{1-z}}%
\right) .  \label{2F1_4/3}
\end{eqnarray}%
Therefore, inserting (\ref{2F1_2/3}) and (\ref{2F1_4/3})\ in (\ref{3F2_a}),
taking into account the property (\ref{Pnumu}),\ and knowing, according to
\cite[Eqn. 43:4:5]{Atlas}, that $\Gamma \left( \frac{2}{3}\right) \Gamma
\left( \frac{4}{3}\right) =\frac{2\pi }{3\sqrt{3}}$, we obtain (\ref%
{Product_P}), as we wanted to prove.
\end{proof}

\section{Conclusions}

We have considered the solution of $x^{n}-x+t=0$ for $n=2,3,4$, both in
terms of hypergeometric functions as well as in terms of elementary
functions. Thereby, we have obtained some reduction formulas of
hypergeometric functions. In order to extend the latter results, we have
applied the differentiation formulas (\ref{Dn_(1)}), (\ref{Dn_(2)}), (\ref%
{Dn_(3)})\ and (\ref{Dn_(4)}), as well as the integration formula stated in (%
\ref{Int_Andrews}). Consequently, we have derived new identities and
infinite integrals involving special functions, i.e. the incomplete beta
function, the lower incomplete gamma function, the parabolic cylinder
function and the Legendre function. All the results presented in this paper
have been tested with MATHEMATICA\ and are available at
\url{https://bit.ly/2PyPz6Y}.

\appendix{}

\section{The solution of the cubic equation}

According to \cite{McKelvey}, in the solution of the depressed cubic
equation:%
\begin{equation}
x^{3}\pm 3mx+2n=0,\quad m>0,  \label{Cubic_depressed}
\end{equation}%
we may distinguish the following cases:\

\begin{description}
\item[Case I] Sign `$+$' in (\ref{Cubic_depressed}). One real root and two
complex roots:\
\begin{eqnarray*}
x_{1} &=&-2\sqrt{m}\sinh \left( \frac{\sinh ^{-1}\left( n\,m^{-3/2}\right) }{%
3}\right) , \\
x_{2,3} &=&\sqrt{m}\left[ \sinh \left( \frac{\sinh ^{-1}\left(
n\,m^{-3/2}\right) }{3}\right) \pm i\sqrt{3}\cosh \left( \frac{\sinh
^{-1}\left( n\,m^{-3/2}\right) }{3}\right) \right] .
\end{eqnarray*}

\item[Case II] Sign `$-$' in (\ref{Cubic_depressed})\ and $n^{2}-m^{3}>0$.
One real root and two complex roots.
\begin{eqnarray}
x_{1} &=&-2\sqrt{m}\cosh \left( \frac{\cosh ^{-1}\left( n\,m^{-3/2}\right) }{%
3}\right) ,  \label{x_cosh} \\
x_{2,3} &=&\sqrt{m}\left[ \cosh \left( \frac{\cosh ^{-1}\left(
n\,m^{-3/2}\right) }{3}\right) \pm i\sqrt{3}\sinh \left( \frac{\cosh
^{-1}\left( n\,m^{-3/2}\right) }{3}\right) \right] .  \notag
\end{eqnarray}

\item[Case III] Sign `$-$' in (\ref{Cubic_depressed})\ and $n^{2}-m^{3}<0$.
Three real roots.
\begin{eqnarray}
x_{1} &=&-2\sqrt{m}\cos \left( \frac{\cos ^{-1}\left( n\,m^{-3/2}\right) }{3}%
\right) ,  \label{x_cos} \\
x_{2,3} &=&\sqrt{m}\left[ \cos \left( \frac{\cos ^{-1}\left(
n\,m^{-3/2}\right) }{3}\right) \pm \sqrt{3}\sin \left( \frac{\cos
^{-1}\left( n\,m^{-3/2}\right) }{3}\right) \right] .  \notag
\end{eqnarray}
\end{description}

\section{The solution of the quartic equation}

According to Descartes solution of the quartic equation \cite{JLQuartic},
the four solutions of the depressed quartic equation:%
\begin{equation}
x^{4}+p\,x^{2}+q\,x+r=0,  \label{Quartic_depressed}
\end{equation}%
are given by:%
\begin{eqnarray}
x_{1,2} &=&\frac{1}{2}\left( -\alpha \pm \sqrt{\alpha ^{2}-4\beta }\right) ,
\label{x_1,2_quartic} \\
x_{3,4} &=&\frac{1}{2}\left( \alpha \pm \sqrt{\alpha ^{2}-4\gamma }\right) ,
\label{x_3,4_quartic}
\end{eqnarray}%
where $\alpha $ is a solution of the \textit{resolvent bicubic equation}:%
\begin{equation}
\alpha ^{6}+2p\alpha ^{4}+\left( p-4r\right) \alpha ^{2}-q^{2}=0,
\label{Resolvent_cubic_def}
\end{equation}%
and%
\begin{eqnarray}
\gamma &=&\frac{1}{2}\left( p+\alpha ^{2}+\frac{q}{\alpha }\right) ,
\label{gamma_def} \\
\beta &=&\frac{r}{\gamma }.  \label{beta_def}
\end{eqnarray}

Note that the resolvent equation can be solved in $\alpha ^{2}$ with the
solution des\-cribed in Appendix $A$.

\bigskip


\end{document}